\documentclass[12pt, reqno]{amsart}

\usepackage{amsmath, amssymb, amsthm, amsfonts, amsrefs, mathtools, array, xfrac, float, multirow, indentfirst, url, enumerate, comment, afterpage, pifont, makecell}
\usepackage{fullpage}
\usepackage{xcolor}
\usepackage{hyperref}
\usepackage{diagbox}
\usepackage{booktabs, siunitx, xcolor, graphicx}
\usepackage{amsmath}
\usepackage{tikz}
\usepackage[outline]{contour}
\usetikzlibrary{decorations.markings}

\tikzset{>=latex} 
\colorlet{myblue}{blue!80!black}

\allowdisplaybreaks

\theoremstyle{plain}
\newtheorem{theorem}{Theorem}
\newtheorem*{theoremnonumBC}{Borel--Carath\'{e}odory Theorem}

\newtheorem{lemma}[theorem]{Lemma}
\newtheorem*{definition}{Definition}

\theoremstyle{definition}

\renewcommand{\Re}{\operatorname{Re}}

\newcommand{\abs}[1]{\left\lvert#1\right\rvert}

\begin{document}

\title{A note on trigonometric polynomials for lower bounds of $\zeta(s)$}
\author{Nicol Leong and Michael J. Mossinghoff}

\address{
    NL: University of New South Wales (Canberra) at the Australian Defence Force Academy, ACT, Australia
}
\email{nicol.leong@unsw.edu.au}
\address{
    MJM: Center for Communications Research, Princeton, NJ, USA
}
\email{m.mossinghoff@ccr-princeton.org}


\begin{abstract}
Non-negative trigonometric polynomials satisfying certain properties are employed when studying a number of aspects of the Riemann zeta function.
When establishing zero-free regions in the critical strip, the classical polynomial $3+4\cos(\theta)+\cos(2\theta)$ used by de la Vall\'ee Poussin has since been replaced by more beneficial polynomials with larger degree.
The classical polynomial was also employed by Titchmarsh to provide a lower bound on $\abs{\zeta(\sigma+it)}$ when $\sigma>1$.
We show that this polynomial is optimal for this purpose.
\end{abstract}

\maketitle

\section{Introduction}\label{sec intro}

In the study of the Riemann zeta function $\zeta(s)$, a particular class of analytic functions arises naturally. This class was described by Jensen \cite{jensen1919investigation} in the following way. Consider an analytic function $f(s)$ on $\mathbb{C}$, which is holomorphic in an open disk $D_r(s_0)$ of radius $r$ centered at $s_0$; in other words, $f(s)$ has a power series with radius of convergence at least $r$. Assume that $f$ is bounded on the disk $D$ (or that $\abs{\Re(f(s))}$ is bounded, or a similar condition), and that we know the value of the function at a point inside the disk, usually the centre, i.e., $f(s_0)$ is known. General statements may then be made regarding the function, such as bounds on its absolute value, or its real or imaginary part, or its derivative, or regarding its zeros or regions where there are no zeros, etc.
Such statements include examples of inequalities proved using Schwarz's lemma, or more generally using the maximum modulus principle. Of these, results like Jensen's formula and the Borel--Carath\'{e}odory theorem make frequent appearances when studying $\zeta(s)$. The reader is referred to \citelist{\cite{jensen1919investigation}\cite{maz2007sharp}\cite{titchmarsh1986theory}} for more details on such topics.

Due to the influence of the maximum modulus principle, when applying such results to a function $f(s)$ from this class, one encounters a term of the form $f(s)/f(s_0)$, and when $f(s) = \zeta(s)$, there is often a need to bound $\zeta(s_0)^{-1}$ where $\sigma_0 >1$. In this paper, we focus on explicit bounds in this case, where it is important to estimate this well.

When $\sigma >1$, one may estimate $|\zeta(s)|^{-1}$ using the trivial bounds \eqref{trivial zetas}, which depend only on $\sigma$. However, when $\sigma$ is very close to $1$, these bounds are poor, due to the presence of the pole of $\zeta(s)$ at $s=1$. This is discussed in greater detail in Section~\ref{details 1}. 
Titchmarsh \cite[\S3]{titchmarsh1986theory} alleviates this by the use of a non-negative trigonometric polynomial. Recall the following definition (see also \cite[\S3.3]{titchmarsh1986theory}).
\begin{definition} \label{def trig poly}
A non-negative trigonometric polynomial with degree $N$ has the form
\begin{equation}\label{trig poly}
   f(\theta) = \sum_{n=0}^N a_n\cos(n\theta),
\end{equation}
where the coefficients $a_n$ are real, $a_0 >0$, $a_N\neq0$, and $f(\theta)\ge 0$ for all real $\theta$.
\end{definition}
Note that such a polynomial is obtained by multiplying together factors which either have the form $(c_1+c_2\cos(\theta))$ or $(c_3\cos(\theta))^2 = (c_3^2 /2)(1+\cos(2\theta))$, the latter factor squared to ensure non-negativity.
Thus it was assumed that $a_0\neq0$ in \eqref{trig poly}.

Titchmarsh showed that the trivial bound used to estimate $|\zeta(s)|^{-1}$ could be replaced by one depending also on $t$, using a simple polynomial of degree $2$,
\begin{equation}\label{classical}
    2(1+\cos(\theta))^2 = 3+4\cos(\theta)+\cos (2\theta).
\end{equation}
This is commonly referred to as the \textit{classical trigonometric polynomial}. We expand more on this in Section~\ref{details 2}.

This was not the first application of this polynomial in number theory. Non-negative trigonometric polynomials are used to establish zero-free regions for $\zeta(s)$.
This was famously pioneered by de la Vall\'ee Poussin \cite{dlVP}, who used \eqref{classical} to show that $\zeta(\sigma +it)\neq 0$ in the region
\begin{equation*}
\sigma \ge 1- \frac{1}{C\log t},
\end{equation*}
with $C=30.4679$ for $t$ sufficiently large. There is a long history of results where much effort has been put towards widening the region by reducing the size of $C$.
Substantial improvements have been found by using (amongst other things) polynomials with higher degree. For instance, a polynomial of degree $8$ is used in \cite{kondrat1977some} to obtain $C=9.547897$, and one of degree $16$ in \cite{MossinghoffTrudgian2015} to obtain $C=5.573412$. 

It may be surprising then that polynomials with larger degree cannot improve the bounds on $|\zeta(\sigma+it)|^{-1}$ using Titchmarsh's method.
In this article, we show that the classical non-negative trigonometric polynomial is optimal, when applied to obtaining such bounds on $\abs{\zeta(s)}^{-1}$ for $\sigma >1$.
We prove the following statement.

\begin{theorem}\label{cor}
Suppose $f(\theta)$ is a non-negative trigonometric polynomial with degree $N\geq2$ as in \eqref{trig poly}, which satisfies Conditions \textbf{I}, \textbf{II}, and \textbf{III} from Section~\ref{secConditions}.
For $\sigma >1$, when using $f(\theta)$ to obtain a lower bound on $|\zeta(\sigma+it)|$ via \eqref{trig ineq general}, then $N$ is even, and
\begin{equation*}
\frac{a_0}{a_1} \ge \frac{3}{4}.
\end{equation*}
Furthermore, the classical polynomial \eqref{classical} is the unique optimal polynomial, up to scaling.
\end{theorem}

\section{A motivating example}\label{details 1}

We illustrate one example where such bounds are required, omitting details that are not relevant to our discussion. For convenience, we recall the following version of the Borel--Carath\'{e}odory theorem \cite[Theorem~14]{hiaryleongyangArxiv}.

\begin{theoremnonumBC}
Let $s_0$ be a complex number, and let $R$ be a positive number possibly depending on $s_0$. 
Suppose that the function $f(s)$ is analytic in a region containing
the disk $|s-s_0|\le R$. Let $M$ denote the maximum of $\textup{Re}\, f(s)$ on the boundary $|s-s_0|=R$.
Then, for any $r\in (0,R)$ and any $s$ such that $|s-s_0|\le r$, we have
\begin{equation*}
|f(s)| \le \frac{2rM}{R-r} + \frac{R+r}{R-r}|f(s_0)|.
\end{equation*}
If in addition $f(s_0)=0$, then 
for any $r\in (0,R)$ and any $s$ with $|s-s_0|\le r$, we have
\begin{equation*}
|f'(s)| \le \frac{2RM}{(R-r)^2}.
\end{equation*}
\end{theoremnonumBC}

Let $s=\sigma +it$ and $s_0=\sigma_0 +it$, where $\sigma_0 := 1+\delta$ with $\delta >0$. Suppose we required an upper bound on $|\zeta(\sigma+it)|^{-1}$. For $0<\sigma <\sigma_0 \le 2$, say, we first obtain a bound on the absolute value of $f(s)$ := $\zeta'(s)/\zeta(s)$ in the region $|s-s_0|\le r$. This is done via the first assertion of the Borel--Carath\'{e}odory theorem, where the usual practice is to take $M$ such that $|\zeta(s)/\zeta(s_0)|\le M$. (See \citelist{\cite{trudgian2015explicit}\cite{hiaryleongyangArxiv}\cite{cully2024explicit}} for more details.)

To obtain such an $M$ for applying the theorem, the two quantities $|\zeta(s)|$ and $|\zeta(s_0)|^{-1}$ are estimated separately. The main difficulty arises due to the need for a uniform bound in some radius $r$ around $s_0$, i.e., we need a bound on $|\zeta(s_0 +re^{i\theta})/\zeta(s_0)|$, due to the effect of the maximum modulus principle in the background of the Borel--Carath\'{e}odory theorem.

When obtaining explicit bounds for $|\zeta(s)|$, the cases $\sigma\le 1$ and $\sigma >1$ are treated differently. For $\sigma \le 1$, very good point-wise upper bounds on $|\zeta(s)|$ occur in the literature (see for example \cite{cully2024explicit, hiarypatelyang2022, hiaryleongyangArxiv}), and these can be made uniform over a range of $\sigma$ by utilising a convexity argument like the Phragm{\'e}n--Lindel{\"o}f Theorem \cite[\S29, \S33]{rademacher1973Topics}. Lower bounds require zero-free regions, which make the constants involved much harder to deal with (see for instance \cite{trudgian2015explicit}).

{When estimating $|\zeta(s)|$ and $|\zeta(s)|^{-1}$ for $\sigma >1$, the trivial bounds 
\begin{equation}\label{trivial zetas}
    |\zeta(s)|\le \zeta(\sigma)\qquad \text{and} \qquad |\zeta(s)|^{-1} \le \frac{\zeta(\sigma)}{\zeta(2\sigma)}
\end{equation}
are optimal in a particular sense.
This so-called optimality is illustrated in \cite[Theorems~8.4, 8.4(A), 8.6, 8.7]{titchmarsh1986theory}, where we see that for $\sigma >1$, the bounds in \eqref{trivial zetas} hold for all values of $t$, while
\begin{equation}\label{trivial optimal}
    |\zeta(s)| \ge (1-\epsilon)\zeta(\sigma)\qquad \text{and} \qquad |\zeta(s)|^{-1} \ge (1-\epsilon)\frac{\zeta(\sigma)}{\zeta(2\sigma)}
\end{equation}
hold for some indefinitely large values of $t$. In addition, the functions $\zeta(s)$ and $\zeta(s)^{-1}$ are unbounded in the open region $\sigma >1$, $t>\delta >0$.

However, we could hope to improve on the trivial bound for $|\zeta(s)|^{-1}$ for $\sigma >1$ by putting a restriction on our region, in terms of a relationship between $\sigma$ and $t$ (e.g., $\sigma = 1+ (\log t)^{-1}$).
This improvement can then be done with a non-negative trigonometric polynomial.

\section{The problem}\label{details 2}

In what follows, we always assume $\sigma>1$. Consider for instance the classical trigonometric polynomial \eqref{classical}. By the usual process (see \cite[\S3.3]{titchmarsh1986theory}), the non-negativity of \eqref{classical}, with the relation
\begin{equation*}
    |\zeta(\sigma +it)| = \exp \left( \sum_{p\,\text{prime}}\sum_{m=1}^\infty \frac{\cos(mt \log p)}{mp^{m\sigma}}\right)
\end{equation*}
leads to $\zeta^3(\sigma)|\zeta^4(\sigma+it)\zeta(\sigma +2it)| \ge 1$, where the exponents correspond to the coefficients of the right side of \eqref{classical}. This then implies
\begin{equation}\label{trig ineq}
    \left| \frac{1}{\zeta(\sigma+it)}\right| \le \zeta(\sigma)^{3/4}|\zeta(\sigma + 2it)|^{1/4},
\end{equation}
and the right side of \eqref{trig ineq} can be estimated easily using existing bounds on $\zeta(s)$. If this is done using trivial bounds in \eqref{trivial zetas}, we see that \eqref{trig ineq} is superior as long as
\begin{equation}\label{compare}
    |\zeta(\sigma +2it)| < \frac{\zeta(\sigma)}{\zeta(2\sigma)^4}.
\end{equation}
The validity of this inequality is not immediately obvious, as the right side of \eqref{compare} decreases to $(6/\pi^2)^4\zeta(\sigma)\approx 0.136\zeta(\sigma)$ as $\sigma\to 1^+$. On the contrary, by \eqref{trivial optimal} with $\epsilon = 1/2$ say, we have that $|\zeta(\sigma+2it)|\ge 0.5\zeta(\sigma)$ for some indefinitely large values of $t$. 

Nevertheless, the remaining factor of $\zeta(\sigma)$ dominates since it approaches infinity as $\sigma\to 1^{+}$. Thus the advantage that \eqref{trig ineq} has over the second bound in \eqref{trivial zetas}, is to shift some weight away from $\zeta(\sigma)$ and transfer it to $|\zeta(\sigma+2it)|$. The latter factor is much easier to deal with, even if $\sigma =1$. See for instance \cite[Theorem 1]{hiaryleongyangArxiv}, where it is shown that $|\zeta(1+it)|\le 1.731\log t/\log\log t$ for $t\ge 3$. Known bounds on $|\zeta(1+it)|$ of various other orders are also detailed in \cite[\S 1]{hiaryleongyangArxiv}.

The next natural question was then to see if we could improve \eqref{trig ineq} by using higher degree trigonometric polynomials as done for zero-free regions. The goal was to shift more weight away from the dominant term $\zeta(\sigma)$, i.e., reduce $3/4$ in \eqref{trig ineq}, at the cost of including more factors $|\zeta(\sigma+nit)|$, for fixed $n>2$, on the right-hand side. It turns out that this is not possible.

\section{Some conditions}\label{secConditions}

For work on zero-free regions of the zeta function, non-negative trigonometric polynomials as in \eqref{trig poly} are required which satisfy $a_n\ge 0$ for each $n$ and $a_1 > a_0$.
Different conditions are required on the coefficients in our problem of interest. 
In our case, the required criteria are
\begin{align*}
    &Condition~\textbf{ I}.\quad\hspace{3.75mm} a_1 > 0, \\
    &Condition~\textbf{ II}.\quad\, \sum_{n=0}^N |a_n| \le 2 a_1, \\
    &Condition~\textbf{ III}.\quad a_n \ge 0 \text{ for all }n\ge 2.
\end{align*}

Before we discuss the need for these conditions, we first state some assumptions for our work. First, we study bounds of size $O(\log t)$ or smaller, more precisely, bounds of the form $O(f(t))$ where $f(Ct) = f(t) + o(f(t))$ for any positive constant $C$.
Second, we are interested in the asymptotic behavior of $\zeta(s)^{-1}$, that is, we seek an improvement that is valid for all $t\geq t_0$, and do not consider potential improvements that are valid only in a finite range. Third, whenever estimating products of terms of the form $|\zeta(\sigma+int)|$ with $n$ a fixed integer, we consider bounds that are independent of $n$, so do not consider the possibility of some terms admitting better bounds than others.

\subsection{Conditions~\textbf{I} and~\textbf{III}}
By the same argument that leads to \eqref{trig ineq}, a non-negative trigonometric polynomial with degree $N$ and coefficients $a_n$ gives rise to the inequality
\begin{equation}\label{trig ineq general}
    \left| \frac{1}{\zeta(\sigma+it)}\right| \le \Big( \zeta(\sigma)^{a_0}|\zeta(\sigma + 2it)|^{a_2}\cdots|\zeta(\sigma+Nit)|^{a_N} \Big)^{1/a_1}.
\end{equation}
The reason for Condition~\textbf{I} is then clear. 
Concerning Condition~\textbf{III}, we can disregard the possibility of negative $a_n$ since for $k>0$, in order to estimate a factor of $|\zeta(\sigma+nit)|^{-k}$ on the right side of \eqref{trig ineq general}, we would need an additional application of the trigonometric polynomial, which increases $a_0$, and this is undesirable as explained at the end of the previous section. Furthermore, if we already have good bounds for such a factor, then we have no need of this method.

\subsection{Condition~\textbf{II}}
Now suppose we want to establish an overall bound of a certain order, say $\zeta(s)^{-1} = O(\log t)$, via \eqref{trig ineq general}. It is clear that each factor on the right side of \eqref{trig ineq general} needs to be at worst of order $\log t$, and asymptotically no larger, so that 
\begin{equation*}
    \frac{1}{a_1}\left(\sum_{n=0}^N |a_n| -a_1\right) \le 1,
\end{equation*}
which is exactly Condition~\textbf{II}. The argument is similar for bounds of different orders. For more information, see for instance \cite[Lemma~8]{cully2024explicit}, in which $\zeta(s)^{-1} = O(\log t/\log\log t)$ is obtained.

For dealing with the right side of \eqref{trig ineq general}, there are existing bounds in the literature where $\zeta(\sigma +int)=O(\log t)$, where $n$ is fixed. These are either obtained with trivial bounds like 
\begin{equation*}
    |\zeta(\sigma+it)| \le \zeta(\sigma) \le \frac{\sigma}{\sigma-1},\qquad \sigma>1,
\end{equation*}
or more involved methods such as Euler--Maclaurin summation (see \cite{backlund1916nullstellen}, \cite[Proposition~A.1]{carneiro2022optimality}).
In the former case, we set the restriction $\sigma = 1+\delta$, for some $\delta >0$, and eventually choose $\delta = 1/\log t$, which will yield $\zeta(s)=O(\log t)$. This method also allows for bounds of other pre-determined orders to be obtained (for $\sigma >1$), using different choices for $\delta$. For instance, a pre-determined order of $\zeta(s) =O(\log t/\log\log t)$ would require choosing $\delta = \log\log t /\log t$. 

Alternatively, there is the possibility of savings by avoiding the trivial bound here and instead using some other machinery, such as exponential sums. The disadvantage there is the difficulty of arriving at a pre-determined order, since the methods there are much more complicated and less amenable to modification.

\section{Best possible polynomial}\label{secBest}

We prove our main results. We begin with an ancillary lemma, which allows us to prove a general result concerning an extremal problem for non-negative trigonometric polynomials. We then specialise this to our problem.

\begin{lemma}\label{seq lemma}
Let $(a_n)$ be a sequence of non-negative real numbers for $2\le n\le N$.
If
\begin{equation}\label{an constraint}
\sum_{n=2}^N a_n(n^2 -1) = 1,
\end{equation}
then $\sum_{n=2}^N a_n \le 1/3$, and the maximum is attained precisely when $a_2 = 1/3$ and $a_n = 0$ for $3\leq n\leq N$.
\end{lemma}

\begin{proof}
Suppose $a_m>0$ for some $m>2$.
Define a sequence $(b_n)$ by
\begin{equation*}
    b_n = \begin{cases}
        a_2 + \frac{1}{3}\sum_{m} a_{m}(m^2 -1) &\qquad \text{if } n=2,\\
        0 &\qquad \text{if } n=m,\\
        a_n &\qquad \text{otherwise.}
    \end{cases}
\end{equation*}
Then $\sum_{n=2}^N b_n (n^2 -1) = 1$ and $\sum_{n=2}^N b_n > \sum_{n=2}^N a_n$, so that $(a_n)$ is not optimal.
Now the set of allowable sequences subject to \eqref{an constraint} is compact for fixed $N$, so an optimal sequence $(a_n)$ exists, and this is therefore necessarily the sequence of the conclusion.
\end{proof}

\begin{theorem}\label{main}
Suppose $f(\theta)$ is a non-negative trigonometric polynomial with degree $N\ge 2$ as in \eqref{trig poly}, which satisfies Conditions \textbf{I} and \textbf{II} from Section \ref{secConditions}. Then
\begin{equation*}
    \frac{a_0}{a_1} \ge \frac{3}{4}.
\end{equation*}
\end{theorem}
\begin{proof}
For simplicity, let $x=\cos(\theta)$ and assume that $a_0 = 1$, for otherwise we can simply force the assumption to be true by re-scaling all coefficients, without affecting the non-negativity condition. Let also $a_1 = 1+r$ with $r> 0$, so our goal is to show that $r \le 1/3$.

Consider the Chebyshev polynomials of the first and second kind, defined respectively by
\begin{equation*}
T_n(\cos(\theta)) = \cos(n\theta) \qquad \text{and}\qquad U_n(\cos(\theta)) = \frac{\sin((n+1)\theta)}{\sin(\theta)}.
\end{equation*}
Applying the former to $f(\theta)$ and substituting $x=\cos\theta$ produces 
\begin{equation*}
    g(x) := 1+ (1+r)x +\sum_{n=2}^N a_n T_n(x) \ge 0
\end{equation*}
for $\abs{x}\leq1$.
In particular, since $T_n(-1)=(-1)^n$, we have
\begin{equation}\label{eqnrupper}
r\le \sum_{n=2}^N (-1)^n a_n,
\end{equation}
while from Condition~\textbf{II} we have
\begin{equation*}\label{r conds}
r\ge \sum_{n=2}^N |a_n|,
\end{equation*}
and so we conclude
\begin{equation}\label{r}
r= \sum_{n=2}^N |a_n|,
\end{equation}
and further $a_n \le 0$ for $n\geq3$ odd, and $a_n \ge 0$ for $n$ even. Therefore,
\begin{equation*}
g(-1) = -r + \sum_{n=2}^N |a_n| =0.
\end{equation*}
We also require $g'(-1) \ge 0$, since otherwise the non-negativity condition is violated.
Recall that  $T'_n(x) = n U_{n-1}(x)$ and $U_{n-1}(-1) = (-1)^{n-1} n$, so
\begin{align*}
g'(-1) &= 1 + r - \sum_{n=2}^N a_n n^2 (-1)^{n} = 1- \sum_{n=2}^N |a_n| (n^2 -1).
\end{align*}
Hence we must maximise $r$ in \eqref{r} subject to the constraint
\begin{equation*}\label{an bound}
\sum_{n=2}^N |a_n| (n^2 -1) \le 1.
\end{equation*}
Clearly we may assume equality here, and apply Lemma \ref{seq lemma}.
We conclude that $N=2$ and $r=1/3$ is optimal.
\end{proof} 

\begin{proof}[Proof of Theorem~\ref{cor}]
Since the optimal polynomial in Theorem~\ref{main} satisfies Condition~\textbf{III}, this follows immediately, except for the statement that $N$ is even.
For this, using \eqref{eqnrupper} and \eqref{r} in the prior proof along with Condition~\textbf{III}, we conclude that $a_n =0$ when $n>2$ is odd, so $N$ must be even.
\end{proof}

\section*{Acknowledgments}

The authors thank the anonymous referee for their helpful comments and suggestions that enhanced the clarity of this work.

\bibliographystyle{amsplain} 
\bibliography{references}

\end{document}